\documentclass[11pt]{article}
\usepackage[english]{babel}

\usepackage[letterpaper,top=2cm,bottom=2cm,left=3cm,right=3cm,marginparwidth=1.75cm]{geometry}

\usepackage{amsmath}
\usepackage{amsfonts}
\usepackage{mathtools}
\usepackage{graphicx}
\usepackage[colorlinks=true, allcolors=blue]{hyperref}
\usepackage[normalem]{ulem}
\usepackage{soul}
\usepackage{xcolor}
\usepackage{amsthm}
\usepackage{blindtext}

\newtheorem{lemma}{Lemma}
\newtheorem{theorem}{Theorem}
\newtheorem{definition}{Definition}

\newtheorem{proposition}{Proposition}

\title{Fundamental group and twisted Alexander polynomial of link complement in 3-torus}
\date{2023\\ March}
\author{Bao Vuong \\ Tomsk State University \\ vuonghuubao@live.com}
\begin{document}

\maketitle

\sloppy

\noindent{\sc Abstract. } We consider a diagrammatic approach to investigate tame knots and links in three dimensional torus $T^3$. We obtain a finite set of generalised Reidemeister moves for equivalent links up to ambient isotopy. We give a presentation for fundamental group of link complement in 3-torus $T^3$ and the first homology group. We also compute Alexander polynomial and twisted Alexander polynomials of this class of links.\medskip


\vspace{10pt}

\begin{center}
Dedicated to Professor Andrey Vesnin on the occasion of his 60th birthday
\end{center}

\section{Introduction}

The dawn of mathematical theory of knots dates back to eighteenth century. It was first treated mathematically in 1771 by Alexandre-Théophile Vandermonde who explicitly noted the importance of topological features when discussing the properties of knots related to the geometry of position. Then it was gradually developed till the early part of twentieth century. The second part of the 20th century was a real golden age for knot theory. The theory became a vast subject, that is ubiquitous in topology and now extending beyond its traditional root in topology to algebraic and differential geometry, number theory, mathematical physic. Dale Rolfsen wrote in his famous book "Knots and Links" \cite{Rolfsen} that the best thing that happened to knot theory, however, is that many more scientists are interested in it, not just topologist, and contributing in their own ways. 
M. Dehn, J. W. Alexander and others studied knots from the point of view of the knot group and invariants from homology theory.
The Alexander polynomial is a knot invariant discovered in 1923 by J. W. Alexander \cite{Alexander}. In technical language, the Alexander polynomial arises from the homology of the infinitely cyclic cover of a knot complement. Any generator of a principal Alexander ideal is called an Alexander polynomial. The Alexander polynomial remained the only known knot polynomial until the Jones polynomial was discovered in 1984. V. Jones's discovery of the new polynomial inspired a "polynomial fever" rampant \cite{Rolfsen}. This breakthrough shed light to our natural world of knots, this led to the discovery of many new polynomials. The news polynomial are just the tips of icebergs, deep results and sophisticated structures are hidden itself under water such as quantum groups and Floer homology. 

Classical knot theory is the study of knots and links in the 3-dimensional sphere $S^3$ or just 3-dimensional Euclidean space $R^3$, that are the simplest 3-manifolds. For the latest years studies on knots and links have been generalised in other spaces as solid torus (see \cite{Berge}, \cite{Gabai1}, \cite{Gabai2}), in projective space (see \cite{Drobotukhina}, \cite{HuynhLe}), in lens spaces (see, for example, \cite{Cattabriga1}, \cite{Cattabriga2}, \cite{Cattabriga3}), in homology 3-sphere (\cite{Moussard}.

In this work we extend further the study in this direction to knots and links in other manifold, namely the three dimensional torus $T^3$, that we shortly call it 3-torus. Some results about links in the space as product of a surface and a circle appear recently making use of skein module theory.
 R. Detcherry and M. Wolff in \cite{Renaud} provide an explicit spanning family for the skein modules, associated with any closed oriented surface. Combined with earlier work of Gilmer and Masbaum \cite{Gilmer1,Gilmer2}, they obtain the dimension of the skein modules for product of a surface and a circle. 
In \cite{Mroczkowski} M. K. Dabkowski and M. Mroczkowsk introduce diagrams and Reidemeister moves for links in $F \times S^1$, where $F$ is an orientable surface.
Using these diagrams they compute the Kauffman Bracket Skein Modules for $D^2 \times S^1$, $A \times S^1$ and $F_{0,3} \times S^1$, where $D^2$ is a disk,  $A$ is an annulus and $F_{0,3}$ is a disk with two holes.

Our modest purpose is using the classical diagrammatic approach to obtain a scheme for computation the fundamental group of link complement in 3-torus and also the first homology group. By doing so we introduce a set of Reidemeister type moves for diagrams of link in 3-torus, which is similar to that of Mroczkowski and Dabkowski in \cite{Mroczkowski}. We are also interested in the twisted Alexander polynomial, using Fox free differential calculus we point out how to compute it for links in 3-torus. We prove some basic properties of the twisted Alexander polynomial for some simple links and local links.

\section{Diagrams}

A link $L$ with $n$ components in three-dimensional torus $T^3$ is an embedding of a disjoint union of $n$ circle  $S^1$ into three-dimensional torus. If $n=1$ the link is called knot. Two link are considered equivalent if they are ambient isotopic, that is, if there exists a continuous deformation of $T^3$ which takes one link to the other.

The three-dimensional torus, or 3-torus, is defined as any topological space that is homeomorphic to the Cartesian product of three circles $T^3=S^1 \times S^1 \times S^1$. The 3-torus is a three-dimensional compact manifold with no boundary. It can be obtained by "gluing" the three pairs of opposite faces of a cube, where being "glued" can be intuitively understood to mean that when a particle moving in the interior of the cube reaches a point on a face, it goes through it and appears to come forth from the corresponding point on the opposite face, producing periodic boundary conditions  (see Figure \ref{figure1}). Thus, the 3-torus is the quotient of a cube $\mathcal{C}$ by the equivalence relation on the boundary $\partial \mathcal{C}$ of the cube which identifies its opposite faces. We denote by $F: \mathcal{C} \rightarrow T^3 = \mathcal{C}/\sim $ the quotient map. Denote by $A$ the bottom face, $A'$ the top face, $B'$ the right face, $B$ the left face, $C$ the front face and $C'$ the back face of the cube $C$.

 \begin{figure}[htbp]
\begin{center}
\includegraphics[scale=0.3]{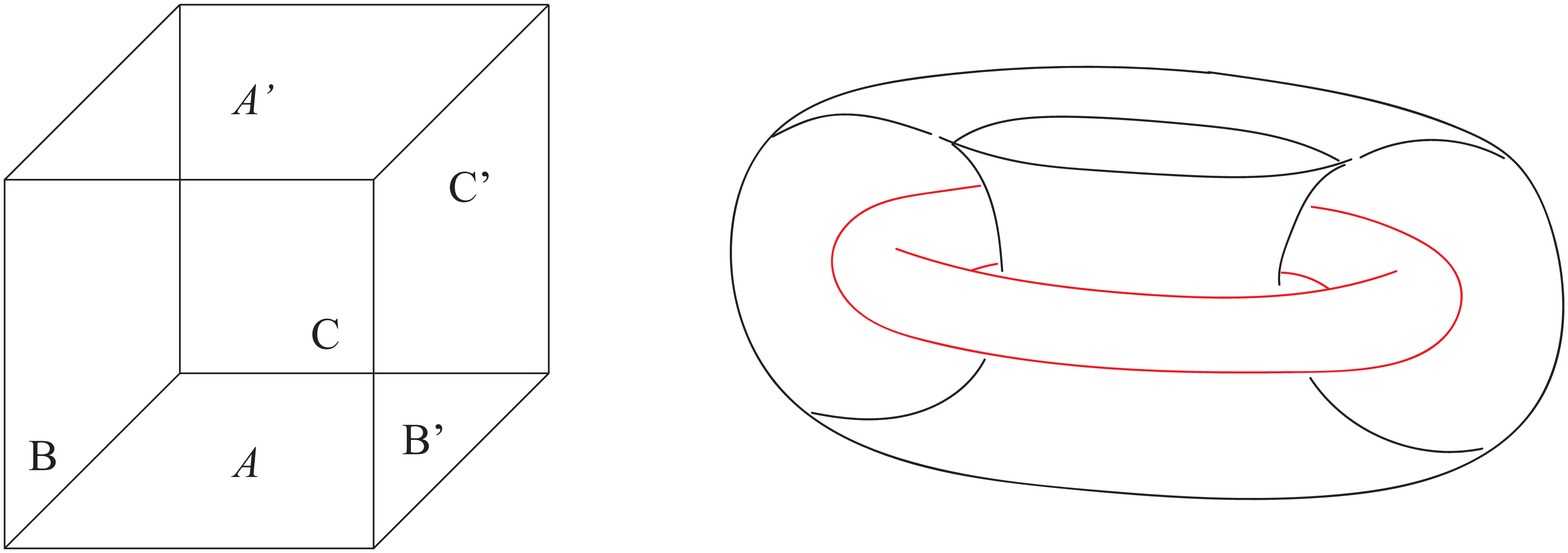}
\end{center}
\caption{3-torus}
\label{figure1}
\end{figure}

We will define diagram for links in 3-torus analogous to that of diagram of links in lens spaces given in \cite{Cattabriga1}. Let $L$ be a link in 3-torus $T^3$ and consider $L'=F^{-1}(L)$. By moving via small isotopy in $T^3$, we can suppose that:

a) $L'$ does not meet the vertices and edges of the cube $\mathcal{C}$;

b) $L' \cap \partial \mathcal{C}$ consists of a finite set of points;

c)$L'$ is not tangent to $\partial \mathcal{C}$;

 \begin{figure}[htbp]
\begin{center}
\includegraphics[scale=0.25]{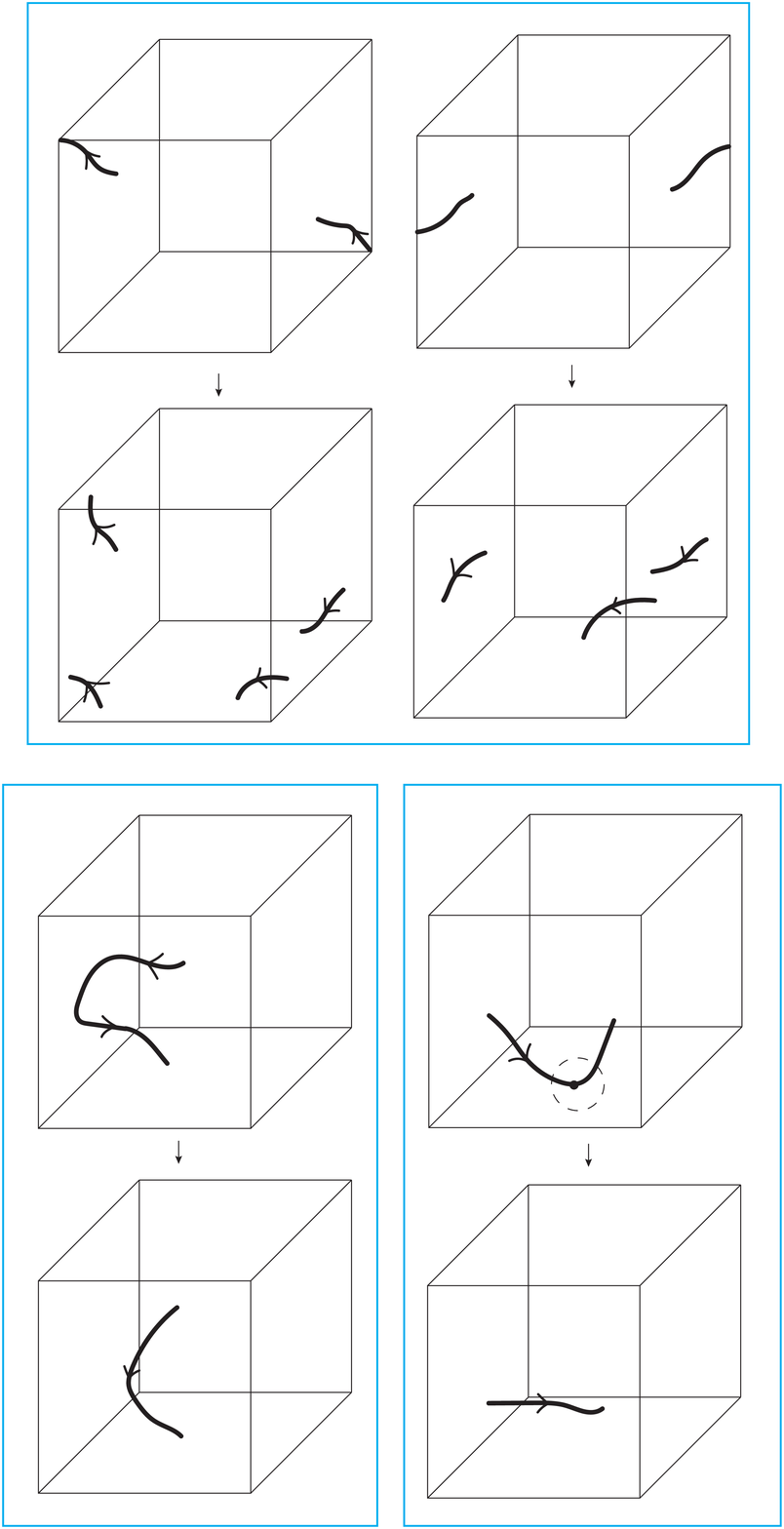}
\end{center}
\caption{Example of small isotopy to reach conditions a), b), c)}
\label{figure2}
\end{figure}

Thus, $L'$ is the disjoint union of closed curves in the interior of the cube $\mathcal{C}$ and arcs properly embedded in $C$, that is only the boundary points belong onto $\partial \mathcal{C}$.

Denote by $A$ the bottom face of the cube. Let $p: \mathcal{C} \rightarrow A$ be the usual orthogonal projection defined by $p(x)=l(x) \cap A$, where $l(x)$ is the line, that is orthogonal to the face $A$ and passing through $x$. Take $L'$ and project it via $p | _{L'}: L' \rightarrow A$. For a point $P \in p(L')$, the preimage $p^{-1}(P)$ may contain more than one point. In this case, we say that $P$ is a multiple point. In particular, if it contains exactly two points, we say that $P$ is a double point. We might assume, by moving L via a small isotopy, that the projection $p|_{L'}: L' \rightarrow A$ of $L$ is regular, that is:

1) the projection of $L'$ contains no cusps;

2) all auto-intersections of $p(L')$ are transversal;

3) the set of multiple point is finite, and all of them are actually double points;

4) no double point is on the edges of bottom face $A$.

We call a double point in the projection of $L'$ in the bottom face $A$ a crossing. As for classical knot diagram we specify   over arcs and under arcs for each crossing relative to the space inside the cube $\mathcal{C}$, that is over and under crossing are defined in the context of bottom and top cube's faces. By doing such projection we forget certain information about the knot in 3-torus. Namely knots in 3-torus can wrap up through faces of the cube, the projection does not carry the information when knots wrap up through the bottom face $A$ and it identification top face. Further we call bottom face the floor and top face the ceiling. For that reason we add vertices with poles (positive and negative) to the projected diagram of knot whenever an arc touches the floor or ceiling, that touched point will be projected to a vertex in knot diagram. We specify the neighbourhood of the vertex is the positive pole if it touches the ceiling and negative pole if it touches the floor (see Figure \ref{diagram}). Note that a vertex always has two different poles as we have supposed in the condition c) above that $L'$ is not tangent to faces $\partial \mathcal{C}$. Also an isolated vertex does not appear since every multiple point in a projection are actually double points. We add a condition of regularity for a projection

5) no vertex at multiple point

Now let $Q$ be a double point, consider $p^{-1}_{|L'}(Q) = {P1,P2}$ and suppose that $P_2$ is closer to bottom face $A$ than $P_1$. Let $U$ be a connected open neighbourhood of $P_2$ in $L'$ such that $p(U)$ contains no other double point and does not meet edges of $A$. We call $U$ underpass relative to $Q$. Every connected component of the complement in $L'$ of all the underpasses (as well as its projection in face $A$) is called overpass.

A diagram of a link $L$ in $T^3$ is a regular projection of $L' = F^{-1}(L)$ on the bottom face $A$, with specified overpasses and underpasses and the projections of the underpasses are not depicted in the diagram (see Figure \ref{diagram}). Thus we have a diagram of knot in a square, the opposite edges of which are identified. 

 \begin{figure}[htbp]
\begin{center}
\includegraphics[scale=0.3]{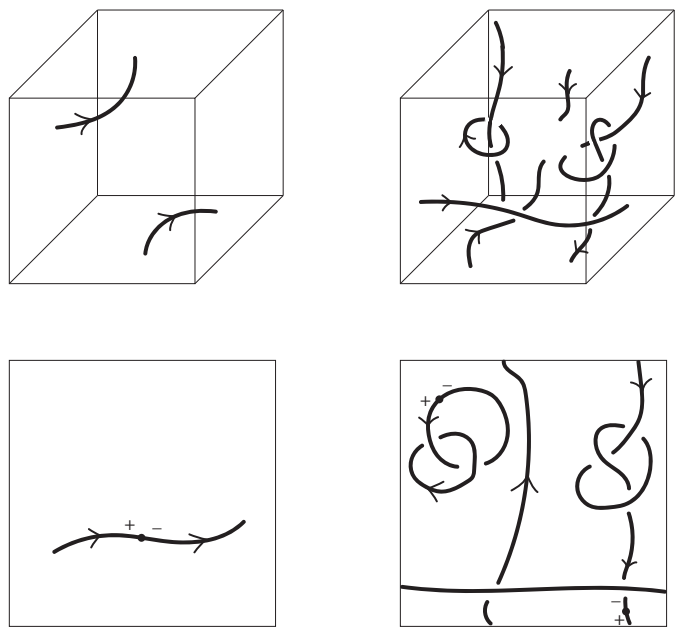}
\end{center}
\caption{Diagram of knot}
\label{diagram}
\end{figure}

\section{Reidemeister moves and vertex moves}

In this section we will have a finite set of moves connecting two different diagrams of the same link. The generalised Reidemeister moves on a diagram of a link $L \subset T^3$, are the moves $R_1, R_2, R_3, R_4, R_5$ (see Figure \ref{rmoves}). The vertex moves are the moves $V_1, V_2, V_3$ (see Figure \ref{vmoves}). Remark that $V_4$ is a forbidden move.

 \begin{figure}[htbp]
\begin{center}
\includegraphics[scale=0.6]{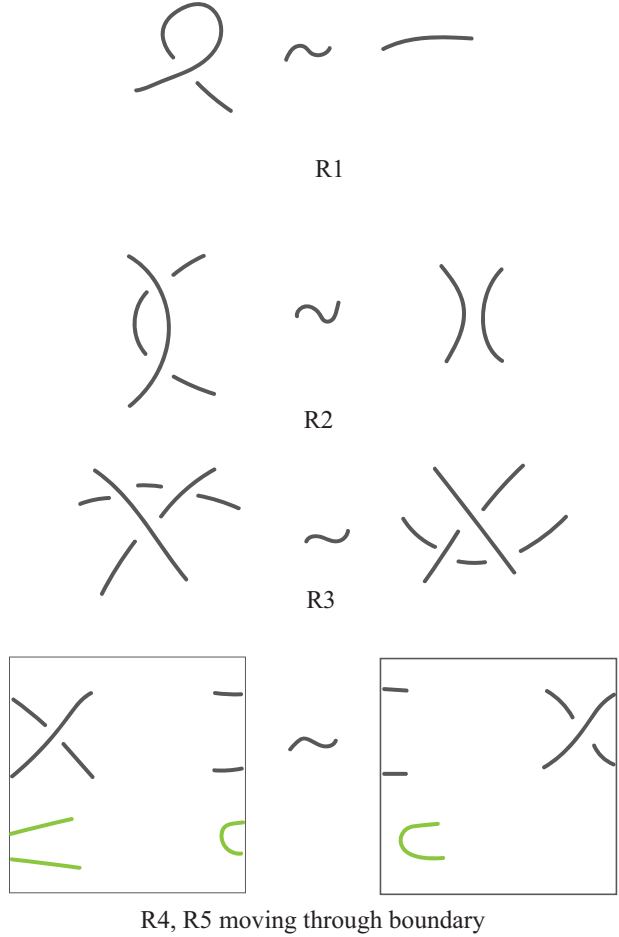}
\end{center}
\caption{Generalised Reidemeister moves}
\label{rmoves}
\end{figure}

 \begin{figure}[htbp]
\begin{center}
\includegraphics[scale=0.35]{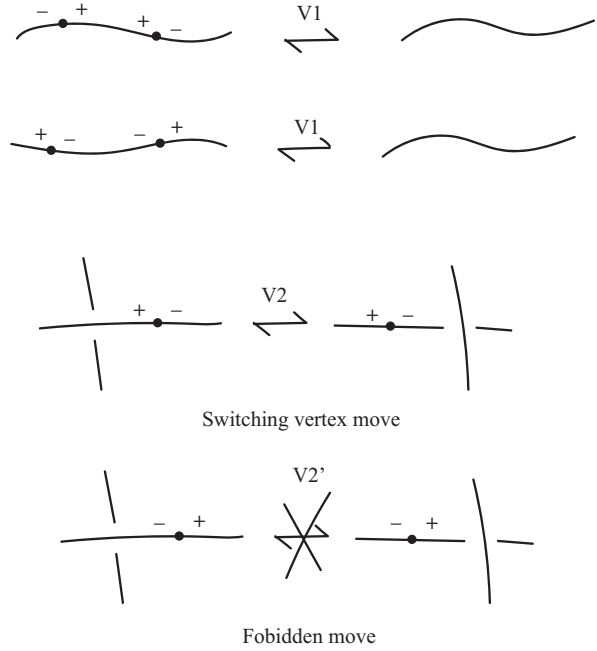}
\end{center}
\caption{Vertex moves}
\label{vmoves}
\end{figure}

\begin{theorem}
Two links $L_0$ and $L_1$ in 3-torus $T^3$ are equivalent if and only if their diagrams can be joined by a finite sequence of generalised Reidemeister moves $R_1, R_2, R_3, R_4, R_5$, vertex moves $V_1, V_2, V_3$ and diagram isotopies.
\end{theorem}

\begin{proof}
It is easy to see that each Reidemeister and vertex move connects equivalent links, hence a finite sequence of the moves and diagram isotopies does not change the equivalent class of the link. Reversely, if we have two equivalent links $L_0$ and $L_1$, then there exists an isotopy of the ambient space $H: T^3 \times [0,1] \rightarrow T^3$ such that $H_0(L_0)= L_0$ and $H_1(L_0) = L_1$. For each $t \in [0,1]$ we have a link $L_t=H_t(L_0)$.

The link $L_t$ may violate conditions a) b) c) and its projections can violate the regularity conditions 1) 2) 3) 4).

It is not hard to see that the isotopy $H$ can be chosen in such a way that conditions b) is always satisfied. We can assume that there are a finite number of forbidden configurations using general position theory (see \cite{Rose}). And for each $t \in [0,1]$, only one forbidden configuration may occur. Figures \ref{Fig6}, \ref{Fig61} illustrates the situations when a condition is violated during the isotopy.

-- conditions 1) 2) 3) generate configurations $S_1,S_2, S_3$.

-- condition c) generates configurations $S_4$ and $P_1$

-- condition 4) generates configuration $S_5$

-- condition 5) generates configuration $P_3$

From each type of forbidden configuration a transformation of the diagram appears, i.e. a generalized Reidemeister move or vertex move, as follows

-- from $S_1,S_2, S_3$ we obtain the usual Reidemeister moves $R_1, R_2, R_3$;

-- from $S_4$ we obtain move $R_4$;

-- from $P_1$ we obtain move $V_1$

-- from $S_5$ we obtain move $R_5$;

-- from $P_2$  we obtain move $V_2$;

 \begin{figure}[htbp]
\begin{center}
\includegraphics[scale=0.3]{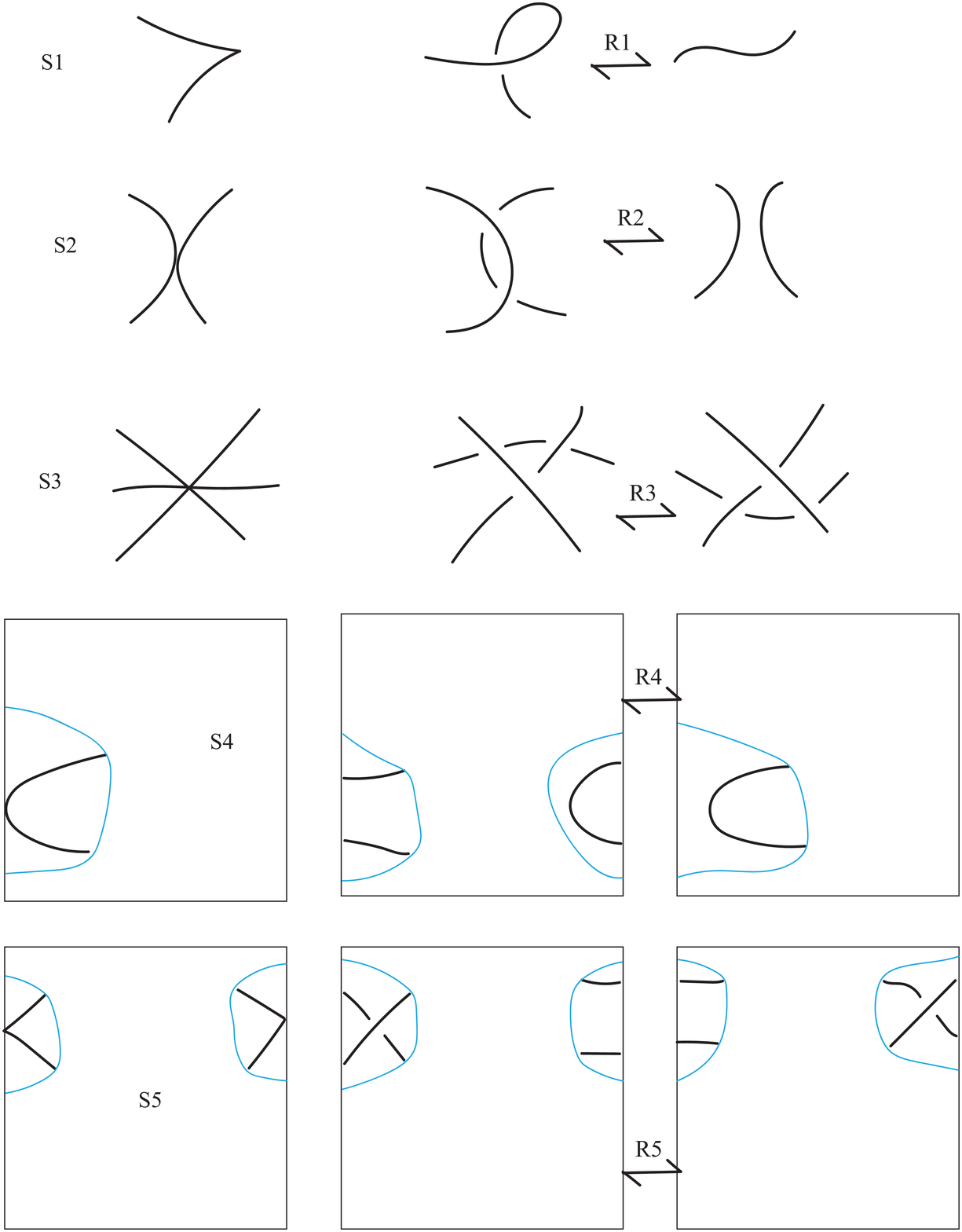}
\end{center}
\caption{Forbidden configurations and diagram moves}
\label{Fig6}
\end{figure}

 \begin{figure}[htbp]
\begin{center}
\includegraphics[scale=0.3]{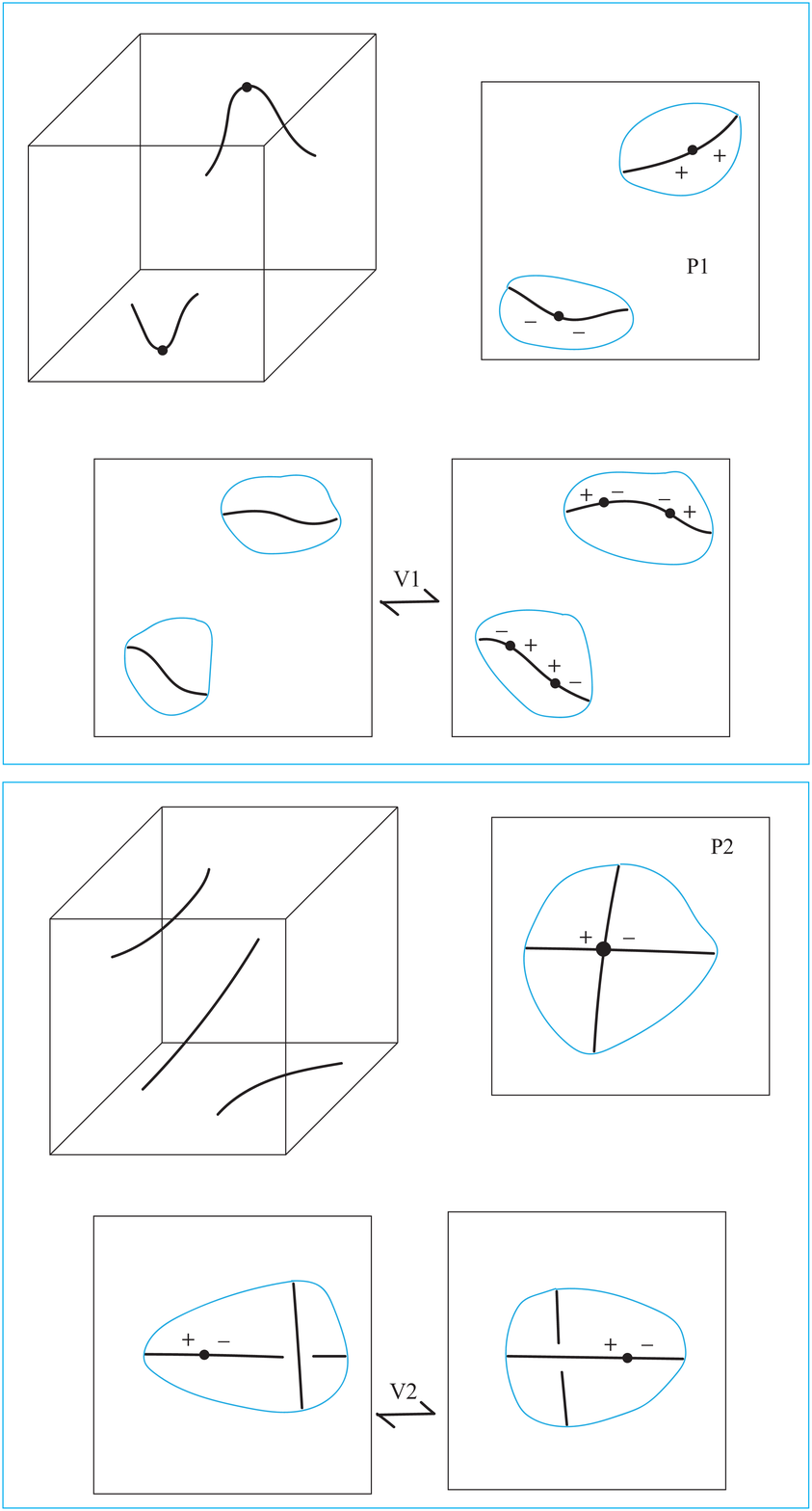}
\end{center}
\caption{Forbidden configurations and diagram moves}
\label{Fig61}
\end{figure}

From the vertex moves $V_1, V_2$ we can obtain two other transformations for a diagram to get a new diagram, that is equivalent to the old one. Namely, the first transformation is the move $V_3$ (see Figure \ref{add}), that describes how we can move a vertex through a double point. The Figure \ref{add} illustrates the case when the moving vertex is in the under arc, the same can be applied when the moving vertex lays in over arc. The second is we can add vertices at intersection of diagram with a Jordan curve, so that the Jordan curve does not meet any double point and the additional vertices have same poles relative to inside and outside of the Jordan curve. Naturally, we can do the reverse process as removing vertices under the same condition.

 \begin{figure}[htbp]
\begin{center}
\includegraphics[scale=0.3]{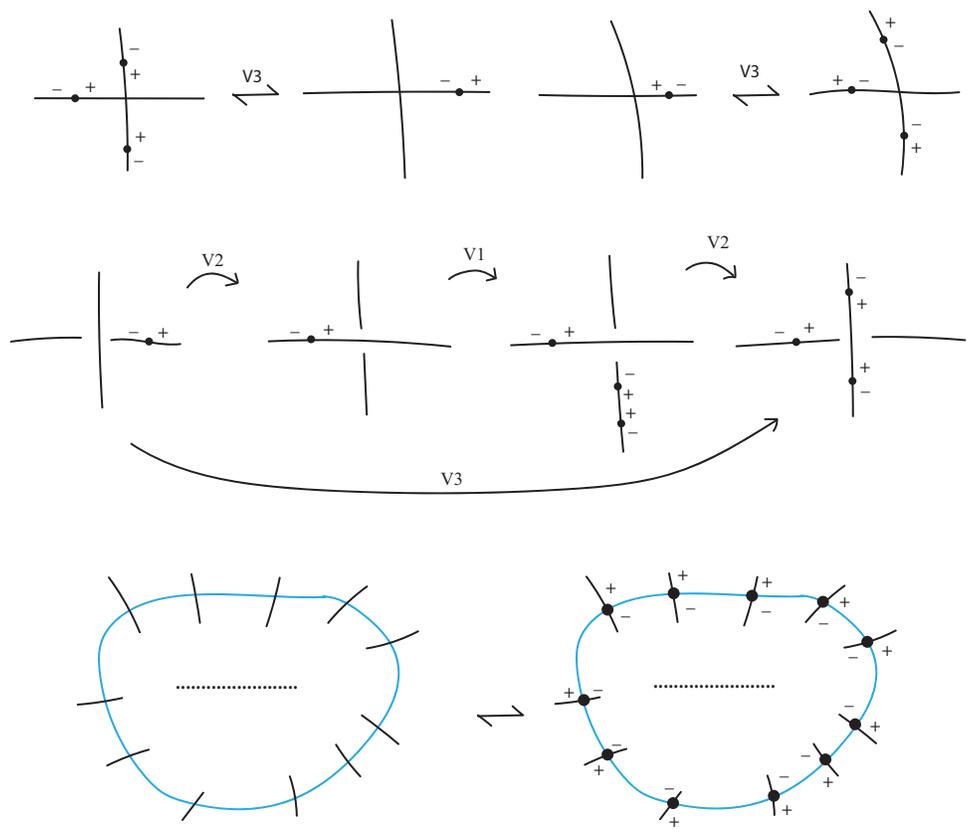}
\end{center}
\caption{Additional moves}
\label{add}
\end{figure}

Thus any pair of diagrams of two equivalent links can be joined by a finite sequence of generalised Reidemeister moves $R_1,..., R_5$; vertex moves $V_1,V_2$ and diagram isotopies.

\end{proof}

\section{Fundamental group}

In this section we obtain, directly from the diagram, a finite presentation for the fundamental group of the complement of links in $T^3$.

Let $L$ be a link in $T^3$, and consider a diagram of $L$. Fix an orientation for $L$, which induces an orientation on both $L'$ and $p(L')$. We call a boundary point is either a vertex or an intersection point of $p(L')$ and the square boundary of the face $A$. Perform an $R_1$ move on each overpass of the diagram having both endpoints on the the same boundary point; in this way every overpass has at most one boundary point. Then label the overpasses as follows: $X_1, ... , X_n$ are the ones ending in the left face $B$; $X_1', ... , X_n'$ are the overpasses ending in the right face $B'$; $Y_1,...,Y_m$ are the overpasses ending in the front face $C$; $Y_1',...,Y_m'$ are the overpasses ending in the back face $C'$; $Z_1,...,Z_l$ are the overpasses ending in the top face $A'$ and $Z_1',...,Z_l'$ are the overpasses ending in the bottom face $A$. The remaining overpasses are labelled by $A_1,...,A_r$. For each $i=1,...,n$, let $\epsilon_i = 1$ if, according to the link orientation, the overpass $X_i$ starts from a point in the left side of the square; otherwise, if $X_i$ ends in the point, let $\epsilon_i=-1$. Also for each $j=1,...,m$, let $\nu_j = 1$ if, according to the link orientation, the overpass $Y_j$ starts from a point in the bottom side of the square; if $Y_j$ ends in the point, let $\nu_j=-1$. And for each $k=1,...,l$, according to the link orientation, let $\tau_k =1$ if the overpass $Z_k$ starts from positive pole of a vertex and $\tau_k =-1$ if the overpass starts from negative pole.

Note that our model is a quotient space modelled the 3-torus, the arrangement of vertices with poles will affect the way how we assign generators of the loops space (the fundamental group of complement), that doesn't occur in the case of Wirtinger presentation for the case of classical knot diagrams. So for eliminating the ambiguity that might occur, we always can transform any diagram of knots in 3-torus via Reidemeister type moves and vertex move, so that every vertices of the diagram locate in the boundary of unique region of the knot diagram. Also positive pole of a vertex always is in the left compare with negative pole of the vertex. Further for computing fundamental group of knot complement in 3-torus, we arrange every vertices with poles to the bottom left region of diagram in the part A of the square; and the part B contains no vertex (see Figure \ref{arr}). Without loss of generality we order the vertices from right to left by increasing index from 1 to $l$. Further we call it vertex arrangement.

 \begin{figure}[htbp]
\begin{center}
\includegraphics[scale=0.35]{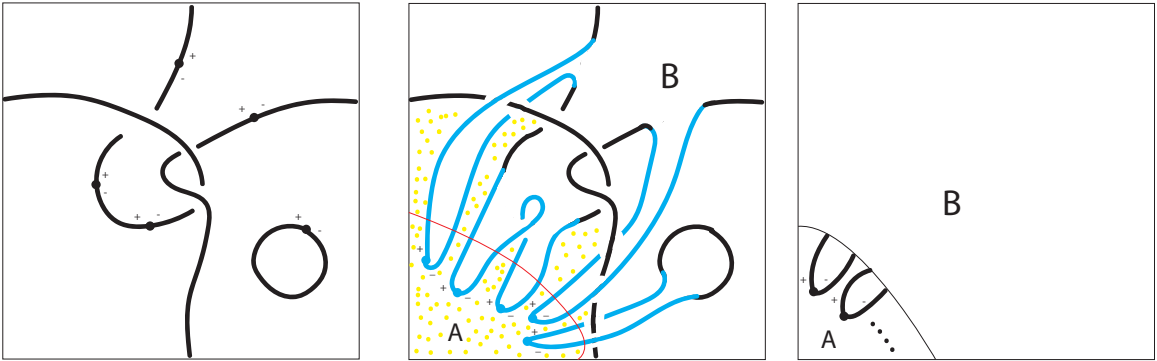}
\end{center}
\caption{Arrangement of vetices}
\label{arr}
\end{figure}

 \begin{figure}[htbp]
\begin{center}
\includegraphics[scale=0.45]{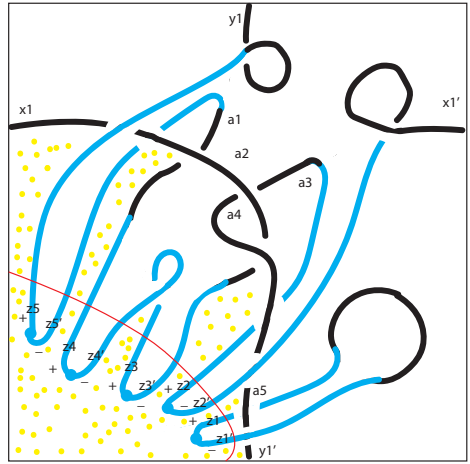}
\end{center}
\caption{Example of overpasses labelling for a link in $T^3$}
\label{Fig8}
\end{figure}

 \begin{figure}[htbp]
\begin{center}
\includegraphics[scale=0.7]{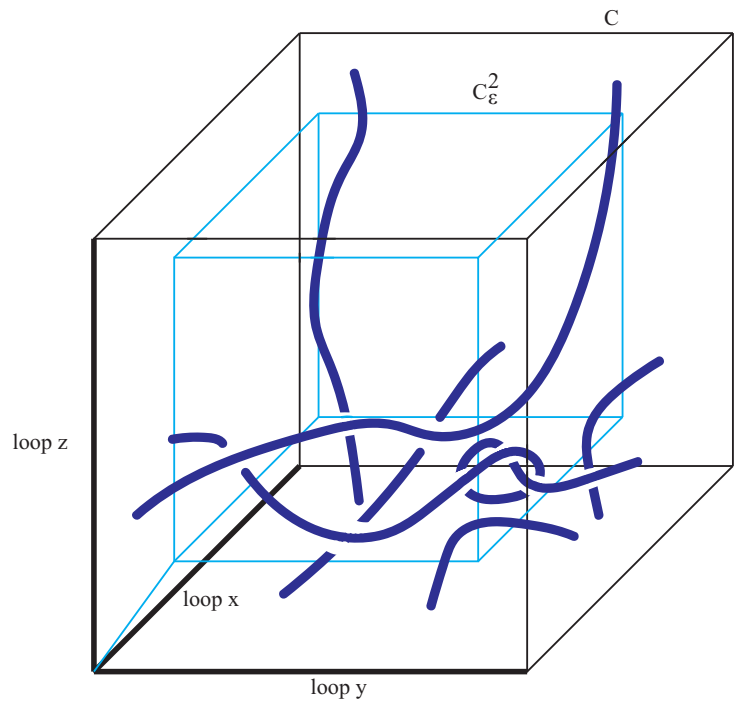}
\end{center}
\caption{A link in 3-torus}
\label{Fig9}
\end{figure}

Associate to each overpass $X_i, Y_j, Z_k, A_q$ a generator $x_i, y_j, z_k, a_q$ respectively, which is a loop around the overpass as in the classical Wirtinger theorem, oriented following the left hand rule.
Moreover let $x,y,z$ be the generators of the fundamental group of the three torus $T^3$ depicted in Figure \ref{Fig9}. Denote by $\gamma_k$ the loop correspond to element $z_{k-1}^{\tau_{k-1}}...z_1^{\tau_1}$ and $\gamma_1=1$. We have the following relations

$\textbf{W}: w_1, . . . , w_s$ are the classical Wirtinger relations for each crossing, that is of the type $a_i a_j a_i^{-1}a_k^{-1} = 1$ or $a_i a_j^{-1} a_i^{-1}a_k^{-1} = 1$, see to Figure \ref{Fig10};

 \begin{figure}[htbp]
\begin{center}
\includegraphics[scale=0.8]{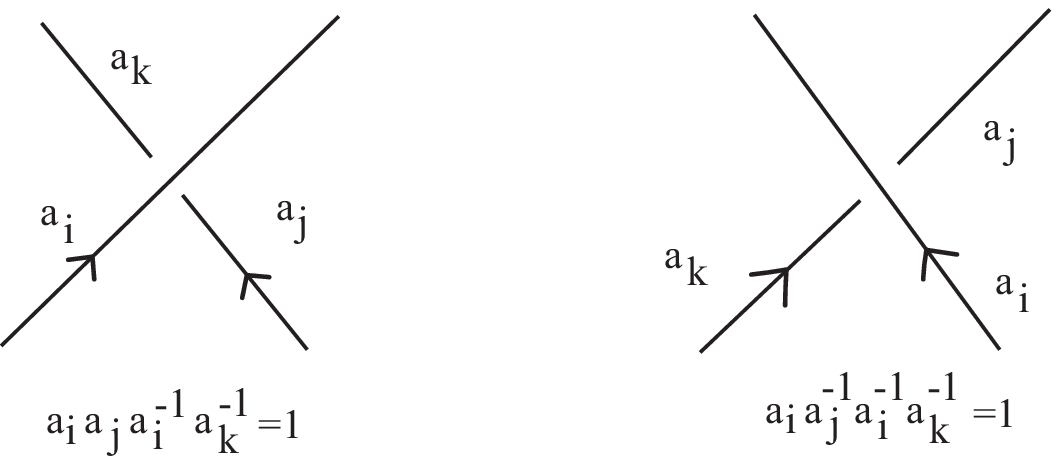}
\end{center}
\caption{Wirtinger relations}
\label{Fig10}
\end{figure}

\textbf{Q:} Relations between loops corresponding to overpasses with identified endpoints on the boundary $x_i'=y\gamma^{-1}_{k+1}x_i\gamma_{k+1}y^{-1}$, $y_j'=x^{-1}\gamma^{-1}_{k+1}y_j\gamma_{k+1}x$, $\gamma^{-1}_{k}z_k'\gamma_{k}=z^{-1}\gamma^{-1}_{k} z_k \gamma_{k} z$

\textbf{T}: Torus relations $x_1^{\epsilon_1}...x_n^{\epsilon_n}=zxz^{-1}x^{-1}$, $y_1^{\nu_1}...y_m^{\nu_m}=yzy^{-1}z^{-1}$, $z_1^{\tau_1}...z_l^{\tau_l}=yxy^{-1}x^{-1}$, when $n, m$ or $l$ is zero then the corresponding product define to be 1.

\begin{theorem}
Let base point be the vertex of the cube, then the group $\pi_1 (T^3 \backslash L)$ of the link $L$ in 3-torus $T^3$ has generators $ x,x_1,..., x_n, x_1',...,x_n', y, y_1,...,y_m,y_1',...,y_m'$, $z, z_1,...,z_l,z_1',...,z_l',a_1,...,a_r $ and relations $W,T$ and $Q$.
\end{theorem}

\begin{proof}
Suppose that $L' = F^{-1}(L)$ is such that $p_{|L'} : L' \rightarrow A$ is a regular projection with the vertex arrangement as above. Consider a cube surface $C^2_{\epsilon}$, so that the maximal distance between the surface of the sphere and the cube boundary $\partial C$ is $\epsilon$. The topological sphere $C^2_{\epsilon}$ divides the cube $\mathcal{C}$ into two pieces: call $I_{\epsilon}$ the internal one and $E_{\epsilon}$ the external one. Choose $\epsilon$ small enough such that all the underpasses belong into int$(I_{\epsilon})$. Let $V, V_{\epsilon}$ be the front bottom left vertices of $\mathcal{C}$ and $C^2_{\epsilon}$ correspondingly (see Fig. \ref{Fig9}). We consider $\tilde{C}^2_{\epsilon}={C}^2_{\epsilon} \cup \overline{VV_{\epsilon}}$ and $\tilde{I}_{\epsilon}=I_{\epsilon}  \cup \overline{VV_{\epsilon}}$.

Now we specify loops, associated with overpasses by the following. Each loop goes along the loop $y$ from the base point and then above every arcs of the knot and winds around the corresponding overpass once and back to the base point. So that the loops have projection to the square, drawn under every vertices in the part $A$ of the square (see Figure \ref{loops}). 

 \begin{figure}[htbp]
\begin{center}
\includegraphics[scale=0.4]{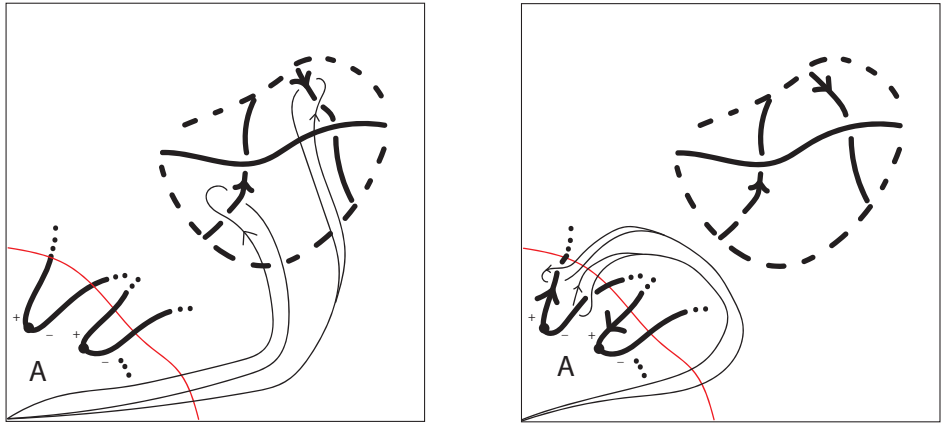}
\end{center}
\caption{Specifying loops}
\label{loops}
\end{figure}

We use Seifert-van Kampen theorem to compute the fundamental group of link complement $\pi_1(T^3 \backslash L,*)$ in three torus $T^3$ with decomposition $(T^3 \backslash L) = (F(\tilde{I}_{\epsilon} \backslash L) \cup (F(E_{\epsilon}\backslash L) $.

The fundamental group of $F(\tilde{I}_{\epsilon} \backslash L)$ can be obtained as in the classical Wirtinger Theorem:

$$\pi_1(F(\tilde{I}_{\epsilon} \backslash L,*)=\langle a_1,...,a_r | w_1, ..., w_s \rangle $$.

For $F(E_{\epsilon}\backslash L)$, we proceed in the following way: first of all observe that we can retract $F(E_{\epsilon}\backslash L)$ to $E \backslash L$, where $E$ is $\partial \mathcal{C} / \sim$. Now the 2 -complex $E$ is a CW-complex consists of: one 0-cell $V$  since all vertices of the cube are identified; three 1-cells correspond to three sets of parallel edges; and three 2-cells correspond to three pair of parallel faces. In order to obtain $\pi_1(E \backslash L)$, we need to add the loops $d_{x_1}, ..., d_{x_n}, d_{y_1}, ..., d_{y_m},d_{z_1}, ..., d_{z_l}$ around the points of $L$ (see Fig. \ref{Fig11}). The relations given by the 2-cell are $d_{x_1}...d_{x_n}=xzx^{-1}z^{-1}$, $d_{y_1}...d_{y_m}=yzy^{-1}z^{-1}$, $d_{z_1}...d_{z_k}=xyx^{-1}y^{-1}$. Hence the fundamental group of $E \backslash L$ is:

\begin{equation*}\label{eq1}
  \begin{gathered}
   \pi_1(E \backslash L,*)=\langle d_{x_1}, ..., d_{x_n}, d_{y_1}, ..., d_{y_m},d_{z_1}, ..., d_{z_l} | d_{x_1}...d_{x_n}=xzx^{-1}z^{-1},\\
    d_{y_1}...d_{y_m}=yzy^{-1}z^{-1}, d_{z_1}...d_{z_l}=xyx^{-1}y^{-1} \rangle
  \end{gathered}
\end{equation*}

 \begin{figure}[htbp]
\begin{center}
\includegraphics[scale=0.6]{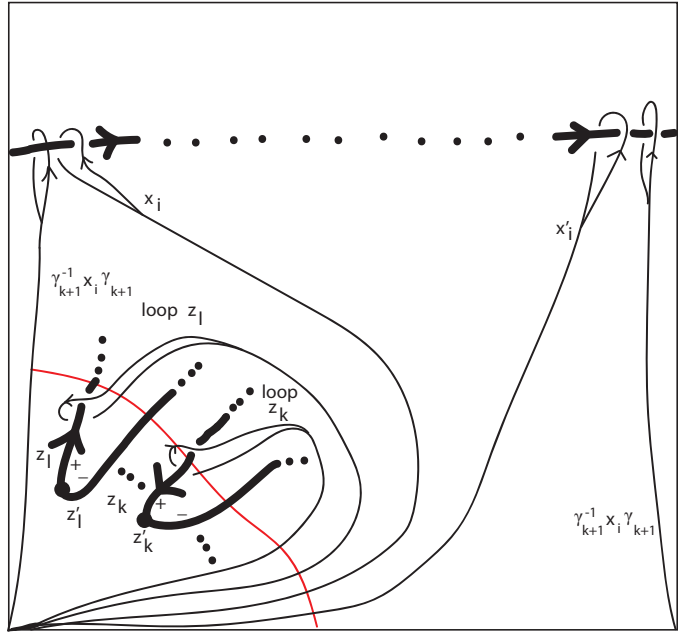}
\end{center}
\caption{Identifying loops $x'_i$}
\label{loopx}
\end{figure}
 
 Finally, the fundamental group of $F(\tilde{C}^2_{\epsilon})\backslash L = (F(\tilde{I}_{\epsilon} \backslash L) \cap (F(E_{\epsilon}\backslash L)$ is generated by $ x_1,..., x_n, x_1',...,x_n', y_1,...,y_m,y_1',...,y_m', z_1,...,z_l,z_1',...,z_l' $. By Seifert-Van Kampen theorem, we need identify each $ x_1,..., x_n, y_1,...,y_m$, $z_1,...,z_l$ with the corresponding generator $d_{x_1}, ..., d_{x_n}, d_{y_1}, ..., d_{y_m}$, $d_{z_1}, ..., d_{z_l}$. Furthermore we need to identify $ x_1',...,x_n',y_1',...,y_m'$, $ z_1',...,z_l' $ with suitable loops in the CW-complex.
 
For the loop $x'_i$, we proceed as follow (see Figure \ref{loopx}), conjugate the loop $x_i$ with $\gamma_{k+1}= z_{k}^{\tau_{k}}...z_1^{\tau_1}$. We get the loop $\gamma^{-1}_{k+1}x_i\gamma_{k+1}$, that winds around the arc $x_i$ once and comes along the loop $x$. Now the loop $x'_i$ can be identified with the loop $y\gamma^{-1}_{k+1}x_i\gamma_{k+1}y^{-1}$. So we have relations $x_i'=y\gamma^{-1}_{k+1}x_i\gamma_{k+1}y^{-1}$ for all $i$.

 \begin{figure}[htbp]
\begin{center}
\includegraphics[scale=0.6]{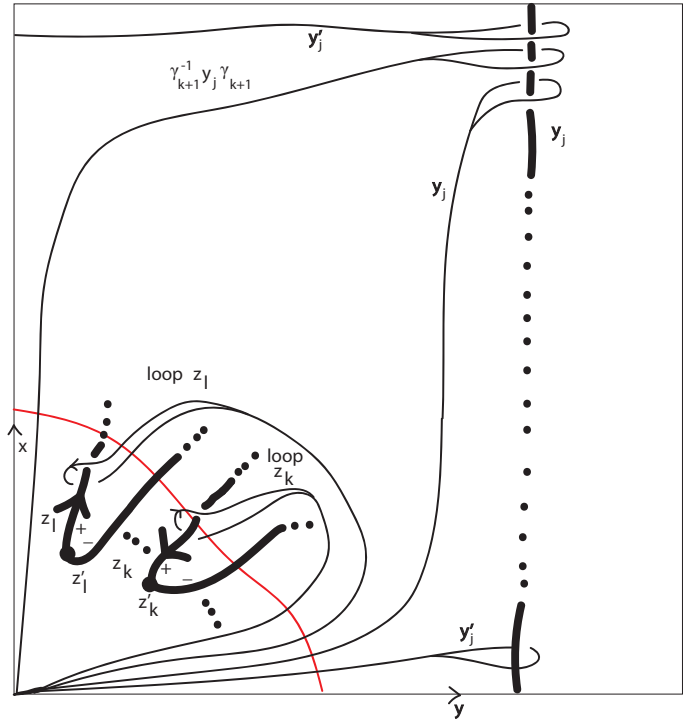}
\end{center}
\caption{Identifying loops $y'_j$}
\label{loopy}
\end{figure}

For the loop $y'_j$, we have analogous situation (see Figure \ref{loopy}), conjugate the loop $y_j$ with $\gamma_{k+1}= z_{k}^{\tau_{k}}...z_1^{\tau_1}$. We get the loop $\gamma^{-1}_{k+1}y_j\gamma_{k+1}$, that winds around the arc $y_j$ once and comes along the loop $x$. Now the loop $y'_j$ can be identified with the loop $x^{-1}\gamma^{-1}_{k+1}y_i\gamma_{k+1}x$. So we have relations $y'_j=x^{-1}\gamma^{-1}_{k+1}y_i\gamma_{k+1}x$ for all $j$. 

 \begin{figure}[htbp]
\begin{center}
\includegraphics[scale=0.6]{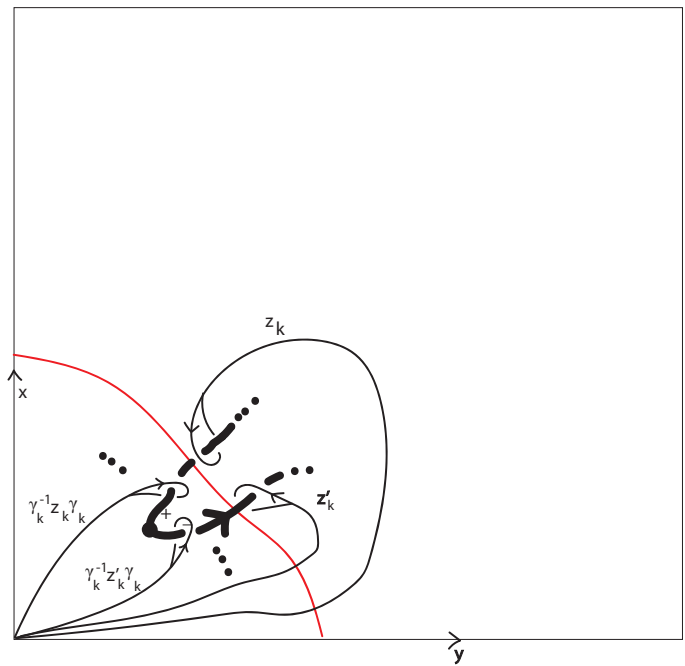}
\end{center}
\caption{Identifying loops $z'_k$}
\label{loopz}
\end{figure}
 
 For the loop $z'_k$ (see Figure \ref{loopz}), conjugate the loop $z_k$ and $z'_k$ with $\gamma_{k}= z_{k-1}^{\tau_{k-1}}...z_1^{\tau_1}$. We get the loop $\gamma^{-1}_{k}z_k\gamma_{k}$ and $\gamma^{-1}_{k}z'_k\gamma_{k}$. Now the loop $\gamma^{-1}_{k}z'_k\gamma_{k}$ can be identified with the loop $z^{-1}\gamma^{-1}_{k}z_k\gamma_{k}z$. So we have relations $z'_k=\gamma_{k}z^{-1}\gamma^{-1}_{k}z_k\gamma_{k}z\gamma^{-1}_{k}$ for all $k$. 
 
 At last we remove $d_{x_1}, ..., d_{x_n}, d_{y_1}, ..., d_{y_m},d_{z_1}, ..., d_{z_l}$ from the group presentation, obtaining a presentation for the fundamental group of link complement in three torus. That has generators $ x,x_1,..., x_n, x_1',...,x_n', y, y_1,...,y_m,y_1',...,y_m', z, z_1,...,z_l,z_1',...,z_l',a_1,...,a_r $ and relations $W,T$ and $Q$.
 
 \end{proof}
 
 \section{The first homology group}
 In this section we show how to determine from the diagram the first homology group of links in 3-torus $T^3$. Consider a diagram of an oriented knot $K \subset T^3$ and let $\epsilon_i, \nu_j, \tau_k$ be as defined in the previous section. Define $\delta = \sum_{i=0}^{n} \epsilon_i$,  $\sigma = \sum_{j=0}^{m} \nu_j$,  $\xi = \sum_{k=0}^{l} -\tau_i$.
 
\begin{lemma}\label{homotype}
If $K \subset T^3$ is an oriented knot and $[K]$ is the homology class of $K$ in $H_1(T^3)$, then $[K]=(\delta, \sigma, \xi)$.
\end{lemma}
 
  \begin{figure}[htbp]
\begin{center}
\includegraphics[scale=0.8]{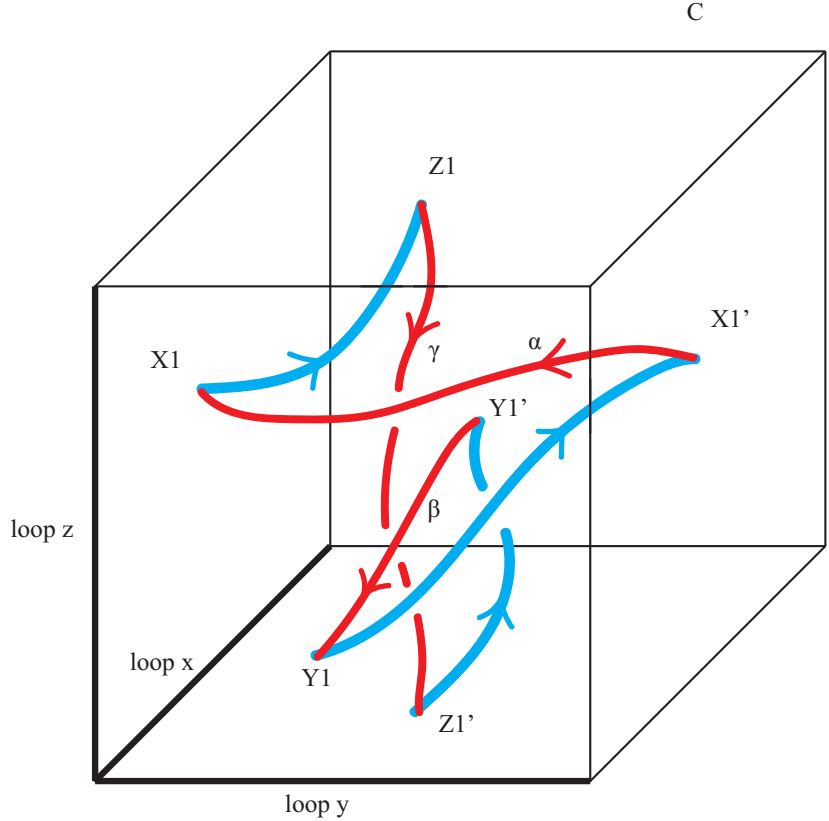}
\end{center}
\caption{Connecting arcs for knot in $T^3$}
\label{Fig12}
\end{figure}
 
 \begin{proof}
 Let $x,y,z$ be generator of $H_1(T^3)=\mathbb{Z} \oplus \mathbb{Z}  \oplus \mathbb{Z}$, as shown in Figure \ref{Fig12}. Let $K \cap (\partial C/\sim) = \{X_1, ... , X_n,X_1', ... , X_n',Y_1,...,Y_m,Y_1',...,Y_m',Z_1,...,Z_l,Z_1',...,Z_l'\}$.
For $i=1,...,n;j=1,...,m;k=1,...,l$ consider the identification class $[X_i]=\{X_i,X_i'\}$; $[Y_j]=\{Y_j,Y_j'\}$; $[Z_k]=\{Z_k,Z_k'\}$. The notations $X_i,X'_i,Y_j,Y'_j,Z_k,Z'_k$ are in the manner as at the section 1 about diagrams. Denote with $\alpha_i$ the loop (red curve) in $T^3$ connecting $X_i$ with $X_i'$ as in Figure \ref{Fig12}, oriented according to the orientation of the knot$K$ as depicted when $\epsilon_i=1$ and in the opposite direction if $\epsilon_i=-1$. Analogously we connect $Y_j$ and $Y_j'$, $Z_k$ and $Z_k'$ by loops $\beta_j$ and $\gamma_k$ respectively. The homology class of these loops are $[\alpha_i]=\epsilon_i y$, $[\beta_j]=\nu_j x$, $[\gamma_k]= -\tau_k z$. The loop $K'$ = $K \cup_{i=1}^{n} \alpha_i \cup_{j=1}^{m} \beta_j \cup_{k=1}^{l} \gamma_k$ is homologically trivial, so we have:

$$(0,0,0)=[K']=[K] + \left(\sum_{i=1}^n[\alpha_i],\sum_{j=1}^m[\beta_j],\sum_{k=1}^l[\gamma_k]\right)$$
 
 Thus, $[K]=(\delta, \sigma, \xi)$, where $\delta = \sum_{i=0}^{n} \epsilon_i$,  $\sigma = \sum_{j=0}^{m} \nu_j$,  $\xi = \sum_{k=0}^{l} -\tau_i$.
 
 \end{proof}
 
 \begin{lemma}\label{homology} Let $L$ be a link in 3-torus $T^3$, with components $L_1,...,L_{\omega}$. For each $\iota=1,...,\omega$, let $(\delta_{\iota},\sigma_{\iota},\xi_{\iota})= [L_\iota] \in \mathbb{Z}^{3}=H_1(T^3)$. Then 
 
  \begin{equation*}
   H_1({T^3 \backslash L }) \cong
    \begin{cases}
       \mathbb{Z}^3 \oplus \mathbb{Z}_\rho, & \text{if   }\omega =1 \\
       \mathbb{Z}^3 \oplus \mathbb{Z}_\kappa \oplus \mathbb{Z}_\lambda,  & \text{if    } \omega = 2 \\
      \mathbb{Z}^{\omega} \oplus \mathbb{Z}_\zeta \oplus \mathbb{Z}_\eta \oplus \mathbb{Z}_\theta, & \text{if    } \omega \geq 3.
    \end{cases}
  \end{equation*}

where  $\rho = gcd({\delta_1, \sigma_1, \xi_1})$;
$\kappa$ and $\lambda$ are the invariant factor of the matrix  $M_1$;
$\zeta, \eta$ and $\theta$ are the invariant factor of the matrix $M_2$.

$$M_1=
\begin{pmatrix}
\delta_1 & \delta_2\\
\sigma_1 & \sigma_2\\
\xi_1 & \xi_2
\end{pmatrix};
$$
$$
M_2=
\begin{pmatrix}
\delta_1 & \delta_2 & ... & \delta_{\omega}\\
\sigma_1 & \sigma_2 & ... & \sigma_{\omega}\\
\xi_1 & \xi_2 & ... & \xi_{\omega}
\end{pmatrix}.
$$

\end{lemma}

\begin{proof}
We abelianize the fundamental group presentation given in previous section. Relation of type $W$ and $Q$ imply that generators corresponding to the same link component are homologous. So $H_1(T^3 \backslash L)$ is generated by $g_1, ..., g_{\omega}$, which are generators corresponding to the link components, and $x,y,z$ generators corresponding to the 3-torus $T^3$. Relation $T$ become:
  
\begin{align}
\delta_1 g_1 + ... + \delta_{\omega} g_{\omega}&=0,\\
\sigma_1 g_1 + ... + \sigma_{\omega} g_{\omega}&=0,\\
\xi_1 g_1 + ... + \xi_{\omega} g_{\omega}&=0.
\end{align}

Thus we have the first homology group $ H_1({T^3 \backslash L }) $ of link complement in 3-torus by Hurewicz's theorem, that's generated by $g_1,...,g_{\omega},x,y,z$ and its relators are (1)(2)(3) as above. If $\omega=1$ then $H_1({T^3 \backslash L }) \cong   \mathbb{Z}^3 \oplus \mathbb{Z}_\rho$, where  $\rho = gcd({\delta_1, \sigma_1, \xi_1})$. If $\omega=2$ then $H_1({T^3 \backslash L }) \cong  \mathbb{Z}^3 \oplus \mathbb{Z}_\kappa \oplus \mathbb{Z}_\lambda$, where $\kappa$ and $\lambda$ are the invariant factor of the matrix  $M_1$. If $\omega \geq 3$ then $H_1({T^3 \backslash L }) \cong \mathbb{Z}^{\omega} \oplus \mathbb{Z}_\zeta \oplus \mathbb{Z}_\eta \oplus \mathbb{Z}_\theta$, where $\zeta, \eta$ and $\theta$ are the invariant factor of the matrix $M_2$. 

\end{proof}

\section{The Alexander-Fox matrix and twisted Alexander polynomials of links in three dimensional torus}

Given a presentation of the group of a link, one may calculate its Alexander polynomial using Fox free calculus \cite{Foxquick}. We recall the following definition of Alexander polynomials (see \cite{Turaev}). Let

$$
P=\langle x_1,...,x_n | r_1,...,r_m \rangle
$$

be a presentation of a group $G$ and denote by $H=G/G'$ its abelianization. Let $F=\langle x_1,...,x_n \rangle$ be the corresponding free group. We apply the chain of maps

$$
\mathbb{Z}F \xrightarrow{\frac{\partial}{\partial x}} \mathbb{Z}F \xrightarrow{\gamma} \mathbb{Z}G  \xrightarrow{\alpha} \mathbb{Z}H,
$$

where $\frac{\partial}{\partial x}$ denotes the Fox differential, $\gamma$ is the quotient map by relations $r_1,...,r_m$ and $\alpha$ is the abelianization map. The Alexander-Fox matrix of the presentation $P$ is the matrix $A=[a_{i,j}]$,  where $a_{i,j}=\alpha(\gamma (\frac{\partial r_i}{\partial x_j}))$ for $i=1,...,m$ and $j=1,...,n$. For $k= 1,..., \text{  min}\{m-1,n-1\}$, the $k$-th elementary ideal $E_k(P)$ is the ideal of  $\mathbb{Z}H$, generated by the determinants of all the $(n-k)$ minors of $A$. The first elementary ideal $E_1(P)$ is the ideal of $\mathbb{Z}H$, generated by the determinants of the all the $(n-1)$ minors of $A$.

\begin{definition}. Let $L \subset S^3$ be a link, and let $E_k(P)$ be the $k$-th elementary ideal, obtained from a presentation $P$ of fundamental group $\pi_1 (S^3 \backslash L, *)$. Then the $k$-th link polynomial $\Delta_k(L)$ is the generator of the smallest principal ideal containing $E_k(P)$. The Alexander polynomial of $L$, denoted by $\Delta(L)$, is the first link polynomial of $L$.
\end{definition}

For a classical link $L$ in $S^3$, the abelianization of $\pi_1 (S^3 \backslash L, *)$ is the free abelian group, whose genereators correspond to the components of $L$. For a link in 3-torus $T^3$, the abelianization of its link group may also contain torsion, as we know by Lemma \ref{homology}. In this case, we need the notion of a twisted Alexander polynomial.

Let $G$ be a group with a finite presentation $P$ and ablianizatioon $H=G/G'$ and denote $K=H/Tors(H)$. Then every homomorphism $\sigma: Tors(H) \rightarrow \mathbb{C}^*=\mathbb{C}\backslash\{0\}$ determines a twisted Alexander polynomial $\Delta^{\sigma}(P)$ as follows. Choosing a splitting $H=Tors(H) \times K$, $\sigma$ defines a ring homomorphism $\sigma: \mathbb{Z}H \rightarrow \mathbb{C} K$ sending $(f,g) \in Tors(H) \times K$ to $\sigma(f)g$. Thus we apply the chain of maps

$$
\mathbb{Z}F \xrightarrow{\frac{\partial}{\partial x}} \mathbb{Z}F \xrightarrow{\gamma} \mathbb{Z}G  \xrightarrow{\alpha} \mathbb{Z}H \xrightarrow{\sigma} \mathbb{C} K
$$

and obtain the $\sigma$-twisted Alexander matrix $A^{\sigma} = \left[ \sigma(\alpha(\gamma (\frac{\partial r_i}{\partial x_j})))\right]$. The twisted Alexander polynomial is then defined by $\Delta ^{\sigma}(P)=\text{gcd} (\sigma(E_1(P)))$.

\begin{definition} Let $L \subset T^3$ be a link in the three dimensional torus $T^3$. For any presentation $P$ of the link group $\pi_1(T^3 \backslash L,*)$, we may define the following.

The Alexander polynomial of $L$, denoted by $\Delta (L)$, is the generator of the smallest principal ideal containing $E_1(P)$.

For any homomorphism $\sigma: Tors (H_1(T^3\backslash L)) \rightarrow \mathbb{C}^*$, the $\sigma$-twisted Alexander polynomial of $L$ is $\Delta^{\sigma}(L) = gcd (\sigma(E_1(P)))$.

\end{definition}

We know from Lemma \ref{homology} that the torsion subgroup of $H_1(T^3 \backslash L))$ is the group $\mathbb{Z}_\zeta \oplus \mathbb{Z}_\eta \oplus \mathbb{Z}_\theta$ in general. So the image of the group homomorphism $\sigma: Tors(H_1(T^3\backslash L)) \rightarrow \mathbb{C}^*$ is contained in the cyclic group, generated by $\Omega$, the $d$-root of unity, where $d$ is $lcm(\zeta,\eta,\theta)$. The $\sigma$-twisted Alexander polynomial $\Delta^{\sigma} (L) \in \mathbb{Z} [\Omega][K]$ is defined up to multiplication by $\Omega^j g$, with $g \in K$.

A link is called local or affine if it is contained in a ball embedded in 3-torus $T^3$. For local links the following properties hold.

\begin{proposition}
Let $L$ be a local link in 3-torus $T^3$. Then the Alexander polynomial $\Delta_L = 0$.
\end{proposition}

\begin{proof}

The fundamental group of $L$ can be presented with the relation of Wirtinger type and the torus relations only. The generator for the group are torus generator $x,y,z$ and the generators corresponding to arcs of the diagram as the one for $L$ in $S^3$. The Jacobian matrix $J$ has the following form

$$
   J= \left[\begin{array}{c c c | c} 
    	y-yxy^{-1}x^{-1} &1-yxy^{-1}& 0 & \\ 
	z-zxz^{-1}x^{-1}    & 0 & 1-zxz^{-1}&0 \\
	  0 & 1- yzy^{-1} & y - yzy^{-1}z^{-1} &\\
    	\hline 
    	 & 0 &  & J_{\bar{L}}
    \end{array}\right] 
$$

where $J_{\bar{L}}$ is the Jacobian matrix of the link group of the link $L$ consider as a link in $S^3$.

The upper left block of the matrix $A$ is the part corresponding to torus generator $x,y,z$ and their relations $yxy^{-1}x^{-1}, zxz^{-1}x^{-1},yzy^{-1}z^{-1}$. For this case the homology group does not contain torsion, so via natural projection to $\mathbb{Z}H$, sending every generators to $t$. We have the Alexander-Fox matrix has the form

$$
   A= \left[\begin{array}{c c c | c} 
    	t-1 &1-t& 0 & \\ 
	t-1    & 0 & 1-t&0 \\
	  0 & 1- t & t - 1 &\\
    	\hline 
    	 & 0 &  & A_{\bar{L}}
    \end{array}\right] 
$$

From this we easily see that $\Delta_L= \Delta_{\bar{L}} \text{det} T$, where $T$ is the matrix corresponding to the upper left block of the matrix $A$. We have $\text{det}T=0$, this implies $\Delta_L = 0$.

\end{proof}

As a consequence a knot in 3-torus with a nontrivial Alexander polynomial cannot be local.

Let $L=L_1 \# L_2$, where $\#$ denotes the connected sum and $L_2$ is a local link. The decomposition $(T^3, L)=(T^3,L_1) \# (S^3,L_2)$ induces monomorphism $j_1: H_1(T^3 \backslash L_1) \rightarrow H_1(T^3 \backslash L)$ and $j_2: H_1(T^3 \backslash L_2) \rightarrow H_1(T^3 \backslash L)$. Given $\sigma: \mathbb{Z} [H_1(T^3 \backslash L)] \rightarrow \mathbb{C} [G]$ induced by $\sigma \in \text{hom}(Tors(H_1(T^3\backslash L)) \rightarrow \mathbb{C}^*)$, denote with $\sigma_1, \sigma_2$ its restrictions to $\mathbb{Z} [j_1(H_1(T^3 \backslash L_1))] $ and  $\mathbb{Z} [j_2(H_1(T^3 \backslash L_2))] $ respectively. We have the following proposition, that is analogous to the Proposition 8 in \cite{Cattabriga1}.

\begin{proposition}
Let $L =L_1\#L_2 \subset T^3$, where $L_2$ is a local link. Then $\Delta^{\sigma} (L)=\Delta^{\sigma_1} (L_1) \cdot \Delta^{\sigma_2} (L_2)$.
\end{proposition}

\begin{proof}
Let $\pi_1 (T^3 \backslash L_1= \langle a_1,...,a_n | r_1,...,r_n \rangle$ is the fundamental group of complement of the link $L_1$ in $T^3$ and $\pi_1 (S^3 \backslash L_2= \langle b_1,...,b_m | s_1,...,s_m \rangle$. Then by the Van Kampen theorem we get a presentation for $\pi_1 (T^3 \backslash L)$ is $\langle a_1,...,a_n, b_1,...,b_m | r_1,...,r_n,s_1, ..., s_m, a_1=b_1 \rangle$. We have the Alexander-Fox matrix of $L$ as follow

$$
   A_L= \left[\begin{array}{c c c c c c c c} 
             &&A_{L_1}&&&&0&\\
             &&0&&&&A_{L_2}&\\
             -1&0&...&0&1&0&...&0
    \end{array}\right] ,
$$

where $A_{L_i}$ is the ALexander-Fox matrix of $L_i$, for $i=1,2$. If $d_k(A)$ denotes the greatest common division of all $k$-minors of a matrix $A$, then by straight forward computation will show that $d_{m+n-1}(A_L) = d_{n-1}(A_{L_1}) \cdot d_{m-1}(A_{L_2})$. So for the case of first elementary ideal of $\pi$, which is the ideal of $\mathbb{C}G$, that implies the equation $\Delta^{\sigma} (L)=\Delta^{\sigma_1} (L_1) \cdot \Delta^{\sigma_2} (L_2)$ for twisted Alexander polynomial.

\end{proof}

\section{Examples}

In this section we carry computation of (twisted) Alexander polynomials for some simple knots in 3-torus.

\textbf{Example 1.} Local trivial knot. The Alexander polynomial of it is 0.

\textbf{Example 2.} Global trivial knot $U_1$, winding along the loop $x$ of the 3-torus once. A diagram for it is depicted in the Fig. \ref{Fig13}. The fundamental group $\pi_1(T^3 \backslash U_1)$ is 

\begin{equation*}\label{eq1}
  \begin{gathered}
 \pi_1(T^3 \backslash U_1)=\langle x,y,z,x_1,x_1'| x_1=x_1', x_1'=y x_1 y^{-1},zxz^{-1}x^{-1}=x_1,\\
 yxy^{-1}x^{-1}=1,yzy^{-1}z^{-1}=1 \rangle = \langle x,y,z, x_1 | x_1yx_1^{-1}y^{-1}=1, \\
 zxz^{-1}x^{-1}=x_1,yxy^{-1}x^{-1}=1,yzy^{-1}z^{-1}=1 \rangle
  \end{gathered}
\end{equation*}

\begin{figure}[htbp]
\begin{center}
\includegraphics[scale=0.5]{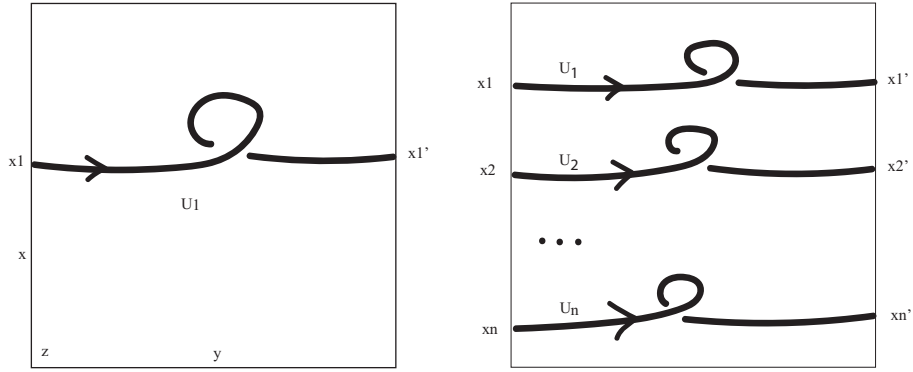}
\end{center}
\caption{Diagrams of $U_1$ (left), $L_n$ (right)}
\label{Fig13}
\end{figure}

Its first homology group is $\langle x,y,z, x_1| x_1=1, xy=yz,xz=zx,yz=zy\rangle = \langle x,y,z | xy=yz,xz=zx,yz=zy \rangle  \cong \mathbb{Z}^3 $. The Jacobian matrix of $\pi_1(T^3 \backslash U_1)$ is

\resizebox{\linewidth}{!}{
$\displaystyle
   A_{U_1}= \left[\begin{array}{c | c c c c } 
           \text{relators}&x&y&z&x_1\\
           \hline
           x_1yx_1^{-1}y^{-1}=1& x_1 - x_1yx_1^{-1}y^{-1}&0&0&1- x_1yx_1^{-1}\\
           zxz^{-1}x^{-1}x_1^{-1}=1&z-zxz^{-1}x^{-1}&0&1-zxz^{-1}&-zxz^{-1}x^{-1}x_1^{-1}\\
           yxy^{-1}x^{-1}=1&y-yxy^{-1}x^{-1}&1-yxy^{-1}&0&0\\
           yzy^{-1}z^{-1}=1&0&1- yzy^{-1}&y-yzy^{-1}z^{-1}&0
    \end{array}\right]
$}

\vspace{10pt}

The first homology group of $T^3 \backslash U_1$ is torsion free, so we have Alexander-Fox matrix is

$$
   A_{U_1}= \left[\begin{array}{ c c c c }
   0&0&0&1-t\\
   t-1&0&1-t&-1\\
   t-1&1-t&0&0\\
  0&1-t&t-1&0
    \end{array}\right] ,
$$

and the Alexander polynomial is $\Delta_{U_1}(t)= (t-1)^2$

\textbf{Example 3.} Global trivial link $L_n$ from $n$ unlinked component $U_1,...,U_n$, all components winding along the loop $x$ of the 3-torus once. A diagram for it is depicted in the Fig. \ref{Fig13}. The fundamental group of $\pi_1(T^3 \backslash L_n)$ is 
\begin{equation*}\label{eq1}
  \begin{gathered}
    \langle x,y,z,x_1,x_1',...,x_n,x_n'| x_1'=y x_1 y^{-1},...,x_n'=y x_n y^{-1},       \\
    zxz^{-1}x^{-1}=x_1...x_n,yxy^{-1}x^{-1}=1,yzy^{-1}z^{-1}=1 \rangle \\
    = \langle x,y,z, x_1,...,x_n | x_1yx_1^{-1}y^{-1}=1,..., x_nyx_n^{-1}y^{-1}=1, \\
    zxz^{-1}x^{-1}=x_1...x_n,yxy^{-1}x^{-1}=1,yzy^{-1}z^{-1}=1 \rangle
  \end{gathered}
\end{equation*}

Its first homology group is commutative additive group $\langle x,y,z, x_1,...,x_n| x_1+ ... + x_n=0 \rangle = \langle x,y,z, x_1,...,x_{n-1} | \rangle  \cong \mathbb{Z}^{n+2} $. The Jacobian matrix $A_{U_1}$ of $\pi_1(T^3 \backslash L_n)$ is

\resizebox{\linewidth}{!}{
$\displaystyle
  \left[\begin{array}{c | c c c c c c c} 
    \text{relators}&x&y&z&x_1&x_2&...&x_n\\
           \hline
      x_1xx_1^{-1}x^{-1}=1& x_1 - x_1xx_1^{-1}x^{-1}&0&0&1- x_1xx_1^{-1}&0&...&0\\
             x_2xx_2^{-1}x^{-1}=1& x_2 - x_2xx_2^{-1}x^{-1}&0&0&0&1- x_2xx_2^{-1}&...&0\\
           ...&...&...&...&...&...&...&...\\
            x_nxx_n^{-1}x^{-1}=1& x_n - x_nxx_n^{-1}x^{-1}&0&0&0&0&...&1- x_nxx_n^{-1}\\
           zxz^{-1}x^{-1}x_n^{-1}...x_1^{-1}=1&z-zxz^{-1}x^{-1}&0&1-zxz^{-1}&-zxz^{-1}x^{-1}x_n^{-1}...x_1^{-1}&-zxz^{-1}x^{-1}x_n^{-1}...x_2^{-1}&...&-zxz^{-1}x^{-1}x_n^{-1}   \\
           yxy^{-1}x^{-1}=1&y-yxy^{-1}x^{-1}&1-yxy^{-1}&0&0&0&...&0\\
           yzy^{-1}z^{-1}=1&0&1- yzy^{-1}&z-yzy^{-1}z^{-1}&0&0&...&0        
    \end{array}\right]
$}

\vspace{+10pt}

The first homology group is torsion free, we get Alexander-Fox matrix, sending generators $x,y,z,x_1,...,x_{n-1}$ to $t$, and $x_n$ to 1. 

$$
   A_{U_1}= \left[\begin{array}{c c c c c c c} 
    0&0&0&1- t&0&...&0\\
    t -1&0&0&0&1- t&...&0\\
    ...&...&...&...&...&...&...\\
    t -1&0&0&0&0&...&1- t\\
    t-1&0&1-t&-t^{-(n-1)}&-t^{-(n-2)}&...&-1\\
    t-1&1-t&0&0&0&...&0\\
    0&1- t&t-1&   0&0&...&0       
    \end{array}\right],
$$

With a simple computation via induction by the number of component, we have the Alexander polynomial for the link $L_n$ is $\Delta _{L_n}=(t-1)^{n+1}$

\section*{Acknowledgments}

The author is grateful to professor Louis Kauffman for his useful comments and conversations, anonymous referee for the verification. This work was supported by the Ministry of Science and Higher Education of Russia (agreement No. 075-02-2023-943).


\bigskip
\begin{thebibliography}{8}

\bibitem{Rolfsen} D. Rolfsen, 
{\em Knots and links}, Mathematics Lecture Series, No. 7. Publish or Perish, Inc., Berkeley, Calif., 1976.

\bibitem{Alexander} J. W. Alexander,
{\em Topological invariants of knots and links},
Trans. Amer. Math. Soc., {\bf 30}: 2 (1928), 275--306.

\bibitem{Berge} J. Berge,
{\em The knots in} $D^2 \times S^1$ {\em with non-trivial Dehn surgery yielding} $D^2 \times S^1$,
Topology Appl., {\bf 38} (1991), 1--19.

\bibitem{Gabai1} D. Gabai,
{\em Surgery on knots in solid tori},
Topology, {\bf 28} (1989), 1--6.

\bibitem{Gabai2} D. Gabai,
{\em 1-bridge braids in solid tori},
Topology Appl., {\bf 37} (1990), 221--235.

\bibitem{Drobotukhina} Y. V. Drobotukhina,
{\em An analogue of the Jones polynomial for links in} $RP^3$ {\em and a generalization of the Kauffman-Murasugi theorem},
Leningrad Math. J., {\bf 2} (1991), 613--630.

\bibitem{HuynhLe} V. Q. Huynh, T. T. Q. Le, 
{\em Twisted Alexander polynomial of links in the projective space},
J. Knot Theory Ramifications, {\bf 17} (2008), 411--438.

\bibitem{Cattabriga1} A. Cattabriga, E. Manfredi, M. Mulazzani,
{\em On knots and links in lens spaces},
Topology Appl., {\bf 160} (2013), 430--442.

\bibitem{Cattabriga2} A. Cattabriga, E. Manfredi, L. Rigolli,
{\em  Equivalence of two diagram representations of links in lens spaces and essential invariants},
Acta Math. Hungar., {\bf  146} (2015), 168--201. 

\bibitem{Cattabriga3} A. Cattabriga, T. Nasybullov,
{\em Virtual quandle for links in lens spaces},
RACSAM, {\bf 112} (2-18), 657--669.

\bibitem{Moussard} D. Moussard,
{\em Finite type invariants of knots in homology 3-spheres with respect to null LP-surgeries},
Geometry \& Topology, {\bf 23}:4 (2019), 2005--2050.

\bibitem{Gilmer1} P. M. Gilmer,
{\em On the Kauffman Bracket Skein Module of the 3-Torus},
Indiana University Mathematics Journal, {\bf 67}:3 (2018), 993--998. http://www.jstor.org/stable/45010317.

\bibitem{Renaud} R. Detcherry, M. Wolff,
{\em A basis for the Kauffman skein module of the product of a surface and a circle}
Algebr. Geom. Topol., {\bf 21}:6 (2021), 2959--2993.

\bibitem{Gilmer2}P. M. Gilmer, G. Masbaum,
{\em On the skein module of the product of a surface and a circle},
Proc. Amer. Math. Soc., {\bf 147} (2019), 4091--4106. 

\bibitem{Mroczkowski} M. Mroczkowski, M.K. Dabkowski,
{KBSM of the product of a disk with two holes and} $S^1$,
Topology and its Applications, {\bf 156}:10  (2009), 1831--1849.

\bibitem{Rose} D. Roseman,
{\em Elementary moves for higher dimensional knots},
Fund. Math., {\bf 184} (2004), 291--310.

\bibitem{Foxquick} R. H. Fox,
{\em  A quick trip through knot theory},
Topology of 3-manifolds, MK, Fort Jr editor, Prentice-Hall, 1962.

\bibitem{Turaev} V.Turaev,
{\em Torsion of 3-dimensional manifolds}, Birkh\"auser Verlag, Basel-Boston-Berlin, 2002.

\end{thebibliography}
\end{document}